\documentclass[a4paper,reqno,11pt]{amsart}
\usepackage{amsmath}
\usepackage{amssymb}
\usepackage{amsthm}
\usepackage{latexsym}
\usepackage[all]{xypic}
\usepackage{amsfonts}
\usepackage{mathrsfs}
\usepackage{graphicx,color}
\usepackage{xcolor}
\usepackage{tikz}
\usepackage{graphics}
\usetikzlibrary[shapes]

\usepackage{anysize,hyperref}

\usepackage{soul}
\setstcolor{red}

\newtheorem{theorem}{Theorem}[section]
\newtheorem{lemma}[theorem]{Lemma}
\newtheorem{corollary}[theorem]{Corollary}
\newtheorem{proposition}[theorem]{Proposition}
\theoremstyle{definition}
\newtheorem{definition}[theorem]{Definition}
\newtheorem{definition-proposition}[theorem]{Definition-Proposition}
\newtheorem{remark}[theorem]{Remark}
\newtheorem{example}[theorem]{Example}

\def\C{\mathcal{C}}
\def\D{\mathcal{D}}
\def\H{\mathcal {H}}

\def\X{\mathscr{X}}
\def\Y{\mathscr{Y}}

\def \text{\mbox}

\providecommand{\add}{\mathop{\rm add}\nolimits}%
\providecommand{\Ext}{\mathop{\rm Ext}\nolimits}%
\providecommand{\Hom}{\mathop{\rm Hom}\nolimits}%

\def\XX{\widetilde{\X}}

\usetikzlibrary{arrows,decorations.pathmorphing,decorations.pathreplacing,positioning,shapes.geometric,shapes.misc,decorations.markings,decorations.fractals,calc,patterns}
\begin{document}

\title{Ptolemy diagrams and torsion pairs in m-cluster categories of type $D$}

\author[Chang]{Huimin Chang}
\address{
Department of Applied Mathematics
School of Education
The Open University of China
100039 Beijing
China
}
\email{changhm@ouchn.edu.cn}

\begin{abstract}

In this paper, we give a complete classification of torsion pairs in $m$-cluster categories of type $D$ when $m$ is odd, denoted by $\C_{D_{n}}^{m}$, via a bijection to combinatorial objects called Ptolemy diagrams of type D. As applications, we classify $m$-rigid subcategories of $\C_{D_{n}}^{m}$, which gives Jacquet-Malo's \cite{J} classification of $m$-cluster tilting subcategories of $\C_{D_{n}}^{m}$. When $m=1$, we generalizes the work of Holm, J{\o}rgensen and Rubey \cite{HJR3} for the classification of torsion pairs in cluster categories of type $D_{n}$.

\end{abstract}

\subjclass[2020]{18E99; 18D99; 18E30}

\keywords{$m$-cluster category of type $D$; torsion pair; Ptolemy diagram}

\thanks{This work was supported by the NSF of China (Grant No.\;12031007). }

\maketitle

\section{Introduction}

Torsion theory is a classical topic in the representation theory of algebras. The notion of torsion pairs in triangulated categories was first introduced by Iyama and Yoshino \cite{IY} to study cluster tilting subcategories in a triangulated category. Cotorsion pairs in a triangulated category were used by Nakaoka \cite{N} to unify the abelian structures arising from $t$-structures \cite{BBD} and from cluster tilting subcategories \cite{KR,KZ,IY}. Torsion pairs and cotorsion pairs in a triangulated category can be transformed into each other by shifting the torsion-free parts (see Definition \ref{dftor}). Hence, classifications of torsion pairs and cotorsion pairs are equivalent in a triangulated category. In fact, classification of (co)torsion pairs in triangulated categories has been studied by many people \cite{CZZ1,CZ,CZ1,GHJ,HJR1,HJR2,HJR3,Ng,ZZZ}.

Cluster categories were introduced with the aim to categorify Fomin and Zelevinsky's cluster algebras, by Buan, Marsh, Reineke, Reiten and Todorov \cite{BMRRT} and also by Caldero, Chapoton and Schiffler who constructed a geometric model for the cluster categories of type $A_n$ \cite{CCS}. As a generalization of cluster categories, Keller \cite{K} introduced $m$-cluster categories with Calabi-Yau dimension $m+1$. Baur and Marsh gave the geometric models of the $m$-cluster categories of type $A_n$ \cite{BM}, and the $m$-cluster categories of type $D_n$ \cite{BM1}, respectively,  which generalise the results of Caldero-Chapoton-Schiffler for $m=1$ of type $A_n$ \cite{CCS}, and Schiffler for $m=1$ of type $D_n$ \cite{S}, respectively.

(Co)Torsion pairs have been classified for many special triangulated categories which have combinatorial models. For example, Ng \cite{Ng} classified torsion pairs in the cluster category of type $A_\infty$ \cite{HJ2} via certain configurations of arcs of the infinity-gon. Holm, J{\o}rgensen and Rubey \cite{HJR1,HJR2,HJR3} classified torsion pairs in the cluster category of type $A_{n}$, in the cluster tube and in the cluster category of type $D_n$ via Ptolemy diagrams of a regular $(n+3)$-gon, periodic Ptolemy diagrams of the infinity-gon and Ptolemy diagrams of a regular $2n$-gon, respectively. Zhang, Zhou and Zhu \cite{ZZZ} classified cotorsion pairs in the cluster category of a marked surface via paintings of the surface. Chang and Zhu \cite{CZ} classified torsion pairs in certain finite $2$-Calabi-Yau triangulated categories with maximal rigid objects, via periodic Ptolemy diagrams of a regular polygon. Chang, Zhou and Zhu \cite{CZZ1} classified cotorsion pairs in the cluster categories of type $A^{\infty}_{\infty}$ via certain configurations of arcs of an infinite strip with marked points  in the plane, called $\tau$-compact Ptolemy diagrams. Gratz, Holm and J{\o}rgensen \cite{GHJ} gave a classification of torsion pairs in Igusa-Todorov's cluster categories of type $A_{\infty}$ via their combinatorial models. Recently, Chang and Zhu \cite{CZ1} classified cotorsion pairs in higher cluster categories of type $A_n$ via Ptolemy diagrams of a regular polygon.

In this paper, we give a classification of torsion pairs in $m$-cluster categories of type $D_{n}$ when $m$ is odd, by using their combinatorial models, but not the one introduced by Baur and Marsh \cite{BM1}, which corresponds to collection of arcs in a punctured regular $(mn-m+1)$-gon. We introduce a new combinatorial model in an unpunctured regular $2(mn-m+1)$-gon. In this model, the indecomposable objects of $m$-cluster categories of type $D_{n}$ are parametrized by pairs of rotationally symmetric arcs and by diameters in two colors, which generalizes the geometric model of the cluster category of type $D_{n}$ given by Fomin and Zelevinsky \cite{FZ}. For a precise description of this model we refer to Section 2 below.

We show in this paper that there is a bijection between the torsion pairs in $m$-cluster categories of type $D_{n}$ and the Ptolemy diagrams of type $D$ in a $2(mn-m+1)$-gon when $m$ is odd. As applications, we classify $m$-rigid subcategories of $m$-cluster categories of type $D_{n}$, denoted by $\C_{D_{n}}^{m}$, which deduces the classification of $m$-cluster tilting subcategories of $\C_{D_{n}}^{m}$ in \cite{J}. When $m=1$, we generalize the work of Holm, J{\o}rgensen and Rubey \cite{HJR3} for the classification of torsion pairs in cluster categories of type $D_{n}$.

The paper is organized as follows. In Section 2, we review the notions of $m$-cluster category of type $D_n$ and its combinatorial model and introduce a new combinatorial model in a regular $2(mn-m+1)$-gon, and review torsion pairs in a triangulated category at last. In Section 3, we define Ptolemy diagrams of type $D_n$ in a $2(mn-m+1)$-gon and give a classification of torsion pairs in $\C_{D_{n}}^{m}$. Some applications are given in the last section.

\subsection*{Conventions} In this paper, $K$ stands for an algebraically closed field. Every subcategory of a category is assumed to be full and closed under taking isomorphisms, finite direct sums and direct summands. Thus, any subcategory is determined by the indecomposable objects in this subcategory. For two subcategories $\X,\Y$ of a triangulated category $\D$, $\Hom_{\D}(\X,\Y)=0$ means $\Hom_{\D}(X,Y)=0$ for any $X\in \X$ and any $Y\in \Y$, and denote by $\X\ast\Y$  the  subcategory of $\D$ whose objects are $M$ which fits into a triangle $X\to M\to Y\to X[1]$ with $X\in\X$ and $Y\in\Y$. A subcategory $\X$ is called  extension closed  provided that $\X\ast\X \subseteq\X.$  We denote by $\X^\perp$ (resp. $^\perp\X$) the subcategory whose objects are $M\in\D$ satisfying $\Hom_{\D}(\X,M)=0$ (resp. $\Hom_{\D}(M,\X)=0$). We use $\Ext_{\D}^i(X,Y)$ to denote $\Hom_{\D}(X,Y[i])$, $i\in\mathbb{Z}$, where $[1]$ is the shift functor of $\D$.

\section{Preliminaries}

\subsection{m-cluster categories}

Let $H$ be a finite dimensional hereditary algebra over a field $K$, and $\H=\D^{b}(H)$ be the bounded derived category of $H$. The Auslander-Reiten translate of $\H$ is denoted by $\tau$ and the shift functor of $\H$ is denoted by $[1]$. We review the notion of $m$-cluster categories  based on \cite{K}, and summarize some known facts about $m$-cluster categories based on \cite{BMRRT,K,Z}.

\begin{definition}\label{def:m-cluster category}
The orbit category $\D^{b}(H)/\tau^{-1}[m]$ is called the $m$-cluster category of $\H$, and is denoted by $\C^{m}(\H)$.
\end{definition}

\begin{definition}
A triangulated category $\D$ is called $d$-Calabi-Yau (shortly $d$-CY) for $d\in\mathbb{N}$ provided there is a functorial isomorphism
\[\Hom_{\D}(X,Y)\simeq D\Hom_{\D}(Y, X[d]),\]
for any $X, Y\in \D$, where $D=\Hom_{K}(-,K)$.
\end{definition}

\begin{proposition}\label{facts:m-cluster category}
\begin{enumerate}
  \item The category $\C^{m}(\H)$ has Auslander-Reiten triangles and Serre functor $S=\tau[1]$, where $\tau$ is the AR-translate and $[1]$ is the shift functor in $\C^{m}(\H)$ induced by the one in $\H$.
  \item The category  $\C^{m}(\H)$ is a Krull-Schmidt $(m+1)$-CY category.
\end{enumerate}
\end{proposition}
When $H=KQ$ is the path algebra over a Dynkin quiver of type $D_{n}$ (see Fig. \ref{D}), the $m$-cluster category $\C^{m}(\H)$, denoted by $\C_{D_{n}}^{m}$, is called the $m$-cluster category of type $D_{n}$, which we consider in this paper. Let $\mathrm{ind}\;\C_{D_{n}}^{m}$ denote the category of isoclasses of indecomposable objects in $\C_{D_{n}}^{m}$.
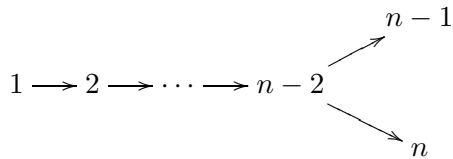
\begin{figure}
\centering

$$
\xymatrix@C=1.5em@R=1em{ &&&&n-1\\
1 \ar@{->}[r] & 2 \ar@{->}[r] & \cdots \ar@{->}[r]& n-2 \ar@{->}[ur] \ar@{->}[dr]\\
&&&& n\\
}
$$
\caption{The Dynkin quiver of type $D_n$}
\label{D}
\end{figure}

\subsection{A geometric model for the m-cluster categories of type $D_{n}$ }
In this subsection, we firstly recall from \cite{BM1} a geometric description of $\C_{D_{n}}^{m}$ in a punctured regular $(mn-m+1)$-gon (let $N=(mn-m+1)$ throughout this paper), then we give a new combinatorial model in a regular $2N$-gon, as Fomin and Zelevinsky's work \cite{FZ} when $m=1$, and show that they are equivalent.

Let $\Pi$ be a punctured regular $N$-gon, $m,n\in\mathbb{N}$, with vertices numbered clockwise from 1 to $N$. We regard all operations on vertices of $\Pi$ modulo $N$ in this subsection. For $i,j\in\{1,2,\cdots,N\}$,
\begin{enumerate}
  \item [(1)] if $i\neq j$, and $j\neq i+1$, i.e. $j$ is not the clockwise neighbor of $i$, an arc $D_{ij}$ is a line from $i$ to $j$ in the clockwise direction.
  \item [(2)] if $i=j$, the two tagged arcs are denoted by $D_{ii}^+$ and $D_{ii}^-$, as in Schiffler's work \cite{S}.
\end{enumerate}
We write $D_{ij}^{\pm}$ (where $j\neq i+1$) to denote an arbitrary arc without confusion, and call it a tagged arc. In that case, if $i\neq j$, then $D_{ij}^{\pm}$ will only stand for the arc $D_{ij}$. As an example, the tagged arcs $D_{14}$, $D_{11}^+$ and $D_{11}^-$ in a punctured regular $9$-gon are shown in Fig. \ref{DD}. 
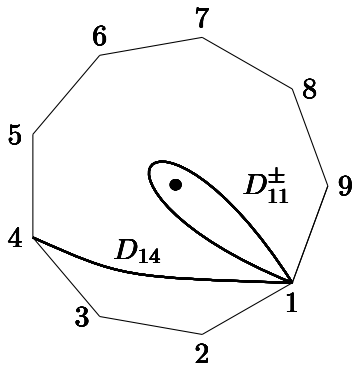
\begin{figure}
\begin{center}
\begin{tikzpicture}[scale=2]
\begin{scope}
    \foreach \x in {-40,0,40,80,120,160,200,240,280,320}{
        \draw (\x:1 cm) -- (\x + 40: 1cm) -- cycle;
        \draw [thick](200:1cm)..controls(240:0.7cm)..(320:1cm);
        \draw [thick](320:1cm)..controls(180:0.8cm) and (100:0.8cm)..(320:1cm);
         \node at (0:0.6cm)   {$D_{11}^{\pm}$};
         \node at (240:0.5cm)   {$D_{14}$};
         \node[right] at (0:1cm) {9};
         \node[right] at (40:1cm) {8};
         \node[above] at (80:1cm) {7};
         \node[above] at (120:1cm) {6};
         \node[left] at (160:1cm) {5};
         \node[left] at (200:1cm) {4};
         \node[left] at (240:1cm) {3};
         \node[below] at (280:1cm) {2};
         \node[below] at (320:1cm) {1};
         \node at (320:0cm) {$\bullet$};

}
\end{scope}
\end{tikzpicture}
\caption{Arcs in a punctured $9$-gon}
\label{DD}
\end{center}
\end{figure}
\begin{definition} \cite[Definition 4.1]{BM1}
Let $D_{ij}$ be an arc of $\Pi$. When $i<j$, we say that $D_{ij}$ is an $m$-arc if $j-i\equiv 1$ mod $m$; When $i>j$, we say that $D_{ij}$ is an $m$-arc if $j+N-i\equiv 1$ mod $m$; When $i=j$, we say that $D_{ii}^\pm$ is an $m$-arc.
\label{Defm}
\end{definition}
\begin{remark}
The definition of an $m$-arc is compatible with Baur and Marsh's in \cite{BM1}.
\end{remark}
\begin{example}
Suppose $m=2, n=5$, and let $\Pi$ be a punctured $9$-gon. Then $D_{14}$, $D_{11}^+$ and $D_{11}^-$ are $2$-arcs, since $4-1\equiv 1$ mod 2 (see Fig. \ref{DD}).
\end{example}
Next we recall the $m$-moves \cite{BM1} generalizing the $m$-rotation of type $A_n$ \cite{BM} and the elementary moves of type $D_n$ \cite{S}.
\begin{definition} \cite[Definition 4.3]{BM1}
Let $\Pi$ be a punctured $N$-gon. An $m$-move arises when there are two arcs with a common endpoint such that the two arcs and a part of the boundary bound an $(m+2)$-gon. We say that there is an $m$-move taking the first arc to the second if the angle from the first arc to the second at the common endpoint is clockwise. More precisely, it is a move of one of the following four forms:
\begin{itemize}
  \item [(1)] $D_{ij}\rightarrow D_{ik}$, if $k-j=m$;
  \item [(2)] $D_{ij}\rightarrow D_{kj}$, if $k-i=m$;
  \item [(3)] $D_{ij}\rightarrow D_{ii}^\pm$, if $i-j=m$;
  \item [(4)] $D_{ii}^\pm\rightarrow D_{ji}$, if $j-i=m$.
\end{itemize}
\label{Defmove}
\end{definition}
The four types of $m$-moves are illustrated in Fig. \ref{moves}.

\begin{figure}
\begin{center}
\begin{minipage}{.22\textwidth}
\centering
\begin{tikzpicture}[scale=1.3]
\begin{scope}
 \foreach \x in {-40,0,40,80,120,160,200,240,280,320}{
        \draw (\x:1 cm) -- (\x + 40: 1cm) -- cycle;
        \draw [->](290:0.65cm)..controls(290:0.7cm)..(300:0.5cm);
        \draw [thick](160:1cm)--(200:1cm);
        \draw [thick](120:1cm)--(160:1cm);
        \draw [thick](200:1cm)..controls(240:0.7cm)..(320:1cm);
        \draw [thick](320:1cm)..controls(200:0.5cm)..(120:1cm);

         \node at (320:0cm) {$\bullet$};

}
\end{scope}
\end{tikzpicture}
\end{minipage}%
\begin{minipage}{.22\textwidth}
\centering
\begin{tikzpicture}[scale=1.3]
\begin{scope}
\foreach \x in {-40,0,40,80,120,160,200,240,280,320}{
        \draw (\x:1 cm) -- (\x + 40: 1cm) -- cycle;
        \draw [<-](90:0.65cm)..controls(85:0.6cm)..(80:0.45cm);
        \draw [thick](240:1cm)--(200:1cm);
        \draw [thick](160:1cm)--(200:1cm);
        \draw [thick](240:1cm)..controls(140:0.5cm)..(40:1cm);
        \draw [thick](160:1cm)..controls(100:0.7cm)..(40:1cm);
         \node at (320:0cm) {$\bullet$};
}
\end{scope}
\end{tikzpicture}
\end{minipage}
\begin{minipage}{.22\textwidth}
\centering
\begin{tikzpicture}[scale=1.3]
\begin{scope}
\foreach \x in {-40,0,40,80,120,160,200,240,280,320}{
        \draw (\x:1 cm) -- (\x + 40: 1cm) -- cycle;
        \draw [->](300:0.65cm)..controls(305:0.6cm)..(306:0.55cm);
        \draw [thick](320:1cm)--(0:1cm);
        \draw [thick](0:1cm)--(40:1cm);
        \draw [thick](40:1cm)..controls(140:0.8cm)and (220:0.9cm)..(320:1cm);
        \draw [thick](320:1cm)..controls(180:0.8cm) and (100:0.8cm)..(320:1cm);
\node at (320:0cm) {$\bullet$};

}
\end{scope}
\end{tikzpicture}
\end{minipage}
\begin{minipage}{.22\textwidth}
\centering
\begin{tikzpicture}[scale=1.3]
\begin{scope}
\foreach \x in {-40,0,40,80,120,160,200,240,280,320}{
        \draw (\x:1 cm) -- (\x + 40: 1cm) -- cycle;
        \draw [->](332:0.55cm)..controls(335:0.6cm)..(340:0.65cm);
        \draw [thick](280:1cm)--(320:1cm);
        \draw [thick](280:1cm)--(240:1cm);

        \draw [thick](240:1cm)..controls(140:0.8cm) and (60:0.8cm)..(320:1cm);
        \draw [thick](320:1cm)..controls(180:0.8cm) and (100:0.8cm)..(320:1cm);
\node at (320:0cm) {$\bullet$};

}
\end{scope}
\end{tikzpicture}
\end{minipage}
\end{center}
\caption{Four types of 2-moves inside a punctured 9-gon}
\label{moves}
\end{figure}
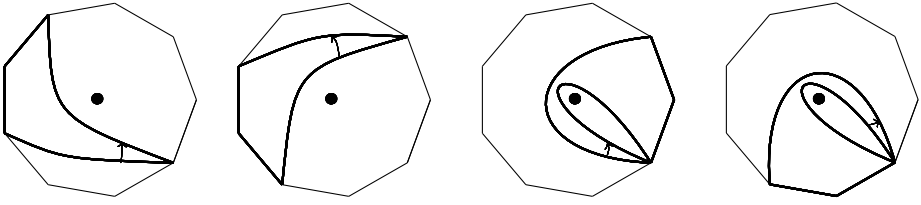

Let $\tau_{m}$ be the translation map acting on $\Pi$ defined as follows. If $i\neq j$ or $m$ is even, then $\tau_{m}$ rotates a tagged arc anti-clockwise around the center. If  $i=j$ and $m$ is odd, then $\tau_{m}$ rotates the tagged arc anti-clockwise around the center and changes its tag. That is,

\begin{equation*}
\tau_{m}(D_{ij}^{\pm})=
\left\{
\begin{aligned}
&D_{i-m,j-m}^{\pm}\quad i\neq j\; \mathrm{or}\; m \;\mathrm{is \;even}\\
&D_{i-m,i-m}^{\mp}\quad i=j \; \mathrm{and}\;  m \; \mathrm{is \;odd}\\
\end{aligned}
\right
.
\end{equation*}

Let $\Gamma(n,m)$ be the quiver whose vertices are the tagged $m$-arcs of $\Pi$, and whose arrows are given by $m$-moves. Then we have the following result by \cite{BM1}.
\begin{theorem}\cite[Theorem 4.5]{BM1}
The quiver $\Gamma(n,m)$ is a translation quiver isomorphic to the AR-quiver of $\C_{D_{n}}^{m}$.
\label{equi}
\end{theorem}
\begin{example}
Suppose $m=2, n=4$. We draw the AR-quiver of $\C_{D_{4}}^{2}$ in a punctured $7$-gon by Theorem \ref{equi} (see Fig. \ref{AR-quiver}).
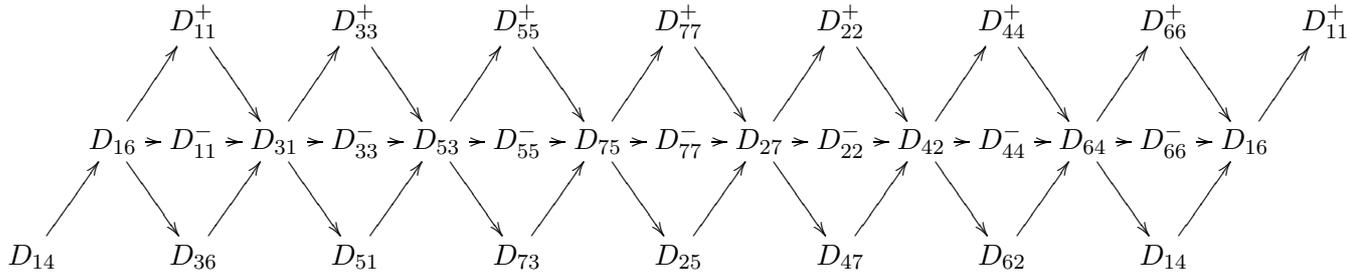
\begin{figure}
\centering
$$
{
\xymatrix@-5mm@C 0.2cm{
    &&D_{11}^{+}\ar[rdd] & & D_{33}^{+}  \ar[rdd] & & D_{55}^{+} \ar[rdd] && D_{77}^{+} \ar[rdd] && D_{22}^{+}\ar[rdd] && D_{44}^{+}\ar[rdd] && D_{66}^{+}\ar[rdd] && D_{11}^{+}\\
 \\
  &D_{16} \ar[ruu]\ar[rdd]\ar[r]&D_{11}^{-}\ar[r]& D_{31} \ar[ruu]\ar[rdd]\ar[r]&D_{33}^{-}\ar[r]& D_{53}\ar[ruu]\ar[rdd]\ar[r]&D_{55}^{-}\ar[r]& D_{75}\ar[ruu]\ar[rdd]\ar[r]&D_{77}^{-}\ar[r]& D_{27}
  \ar[ruu]\ar[rdd]\ar[r]&D_{22}^{-}\ar[r]& D_{42}\ar[ruu]\ar[rdd]\ar[r]&D_{44}^{-}\ar[r]& D_{64} \ar[ruu]\ar[rdd]\ar[r]&D_{66}^{-}\ar[r]& D_{16}\ar[ruu]\\
  \\
  D_{14} \ar[ruu]&& D_{36} \ar[ruu] && D_{51} \ar[ruu]&&D_{73} \ar[ruu]&& D_{25} \ar[ruu]&& D_{47}\ar[ruu] &&D_{62} \ar[ruu] &&D_{14}\ar[ruu]&&\\
\\
}
}
$$
\caption{The AR-quiver of  $\C_{D_{4}}^{2}$}
\label{AR-quiver}
\end{figure}

\end{example}
Now we are ready to give a new combinatorial model in a regular $2N$-gon without puncture, as in Fomin and Zelevinsky's work \cite{FZ} when $m=1$, and establish a connection with Baur and Marsh's combinatorial model \cite{BM1}.

Let $P$ be a regular $2N$-gon. We label the vertices of $P$ clockwise by $1,2,\ldots, 2N$ consecutively. In our arguments below vertices will also be numbered by some $r\in\mathbb{N}$ which might not be in the range
$1\le r\le 2N$. In this case the numbering of vertices always has
to be taken modulo $2N$.

We define an  arc in $P$ to be a set $\{i,j\}$ of vertices of
$P$ with $j\not\in \{i-1,i,i+1\}$, i.e.
$i$ and $j$ are different and non-neighboring vertices. For the sake of convenience, we denote the arc $\{i, j\}$ by $(i, j)$, or just $ij$ without confusion in this section.
There are two types of arcs in $P$: diameter, i.e. the arc connecting two opposite vertices $i$ and $i+N$, and  non-diameter, which is the form $(i,j)$ where $j\in [i+2,i+N-1]$. Note that $[i+2,i+N-1]$ stands for the set of vertices
of $P$ which are met when going clockwise from $i+2$ to $i+N-1$ on the boundary of $P$.

We need two different copies of diameters and denote them by $(i,i+N)_g$ and $(i,i+N)_r$ respectively, where $1\le i\le 2N$. The indices indicate that these diameters are coloured in the colours green and red. For better visibility, we draw the red diameters in a wavelike form and the green ones as straight lines. We use the same notations as in \cite{HJR3}.
\begin{example}
Let $m=2, n=4$. Then $2N=14$. It is easy to see $16, 38, 36,$ $(1,10), (8,13)$ and $(10,13)$ are non-diameter arcs, both $(1,8)_g$ and $(1,8)_r$ are diameter arcs. (See Fig. \ref{14-gon}).
\end{example}
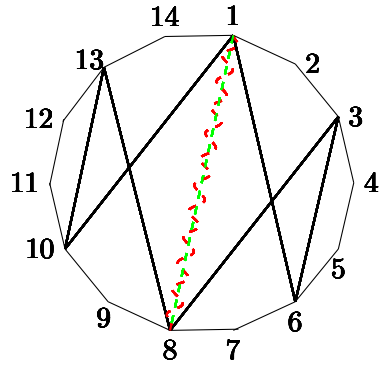
\begin{figure}
\begin{center}
\begin{tikzpicture}[scale=2]
\begin{scope}
    \foreach \x in {-26,0,26,52,78,104,130,155,180,206,232,258,282,308}{
        \draw (\x:1 cm) -- (\x + 26: 1cm) -- cycle;
        \draw [thick](78:1cm)--(308:1cm);
        \draw [thick](26:1cm)--(258:1cm);
        \draw [thick](78:1cm)--(206:1cm);
        \draw [thick](258:1cm)--(130:1cm);
        \draw [thick](26:1cm)--(308:1cm);
        \draw [thick](206:1cm)--(130:1cm);
        \draw (78:1) edge[thick, dashed, decorate, decoration=snake, color=red!] (258:1);
\draw (78:1) edge[thick, dashed, color=green!] (258:1);

         \node[right][below] at (-26:1cm) {5};
         \node[right] at (0:1cm) {4};
         \node[right] at (26:1cm) {3};
         \node[right] at (52:1cm) {2};
         \node[above] at (78:1cm) {1};
         \node[above] at (104:1cm) {14};
         \node[left] at (125:1cm) {13};
         \node[left] at (155:1cm) {12};
         \node[left] at (180:1cm) {11};
         \node[left] at (206:1cm) {10};
         \node[below] at (230:1cm) {9};
         \node[below] at (258:1cm) {8};
         \node[below] at (282:1cm) {7};
         \node[below] at (308:1cm) {6};

}
\end{scope}
\end{tikzpicture}
\caption{Arcs in a regular $14$-gon}
\label{14-gon}
\end{center}
\end{figure}

Next, we give the definition of $m$-arcs in  $P$.
\begin{definition}\label{A}
Let $P$ be a regular $2N$-gon and $(i,j)$ be an arc in $P$. When $i\neq j$, we say that $(i,j)$ is an $m$-arc if $j-i\equiv 1$ mod $m$; When $i=j$, we say the diameters $(i,i+N)_r$ and $(i,i+N)_g$ are $m$-arcs for arbitrary $i=1,2,\cdots,N$.
\end{definition}
\begin{example}
Let $m=2, n=4$ and $P$ be a regular $14$-gon. (See Fig. \ref{14-gon}).
The arc $16$ is a $2$-arc, because $6-1\equiv 1$ mod 2.
Similarly, we have $(3,6), (3,8), (1,10), (8,13)$ and $(10,13)$ are $2$-arcs, the diameters $(1,8)_g$ and $(1,8)_r$ are also $2$-arcs.
\end{example}

Note that every non-diameter arc $(i,j)$ in $P$ has a partner arc $(i+N,j+N)$, which is obtained by a rotation of 180 degrees. The partner arc of a diameter is still itself. If $(i,j)$ is an $m$-arc, so is $(i+N,j+N)$. We denote the pair of arcs $\{(i,j),(i+N,j+N)\}$ by $\overline{(i,j)}$. For example, $\overline{(1,6)}=\overline{(8,13)}$, $\overline{(3,6)}=\overline{(10,13)}$, $\overline{(3,8)}=\overline{(1,10)}$, and $\overline{(1,8)_r}=\overline{(1,8)_r}$, see Fig. \ref{14-gon}.

\begin{definition}
Let $P$ be a regular $2N$-gon. If there are two $m$-arcs in $P$ with a common endpoint such that the two $m$-arcs and a part of the boundary form an $(m+2)$-gon, we say there is an $m$-move taking the first $m$-arc to the second in the clockwise direction.
\end{definition}

\begin{example}
Let $m=2, n=4$, and $P$ be a regular $14$-gon. (See Fig. \ref{14-gon}).
\begin{itemize}
  \item [(1)] There is a  $2$-move $16\rightarrow 36$.
  \item [(2)] There is a  $2$-move $36\rightarrow 38$.
  \item [(3)] There are two  $2$-moves $16\rightarrow 18_r$ and $16\rightarrow 18_g$.
  \item [(4)] There are two  $2$-moves $18_r\rightarrow 38$ and $18_g\rightarrow 38$.
\end{itemize}

\end{example}
 The action of $\tau_m$ on a non-diameter $m$-arc is rotation by $m$ vertices in the counterclockwise direction, i.e. $\tau_m(i,j)=(i-m,j-m)$. The action of $\tau_m$ on diameters is rotation by $m$ vertex in the counterclockwise direction if $m$ is even, and rotation by $m$ vertex in the counterclockwise direction and changing their colour if $m$ is odd, i.e. $\tau_m(i,i+N)_{g,r}=(i-m,i+N-m)_{g,r}$ if $m$ is even, and $\tau_m(i,i+N)_{g,r}=(i-m,i+N-m)_{r,g}$ if $m$ is odd.

Let $\Delta(n,m)$ be the quiver whose vertices are the paired $m$-arcs $\overline{(i,j)}$, and whose arrows are given by $m$-moves. Then we have the following result.
\begin{proposition}\label{Pro1}
$\Delta(n,m)$ is a translation quiver isomorphic to $\Gamma(n,m)$. Moreover, $\Delta(n,m)$ is isomorphic to the AR-quiver of $\C_{D_{n}}^{m}$.
\end{proposition}

\begin{proof}
Let $\varphi:\Gamma(n,m)\rightarrow\Delta(n,m)$ be a morphism of quivers. Suppose $D_{ij}$ is a tagged $m$-arc. If $i<j$, we define $\varphi(D_{ij})=\overline{(i,j)}$; If $i>j$,  we define $\varphi(D_{ij})=\overline{(i,j+N)}$; If $i=j$, we define $\varphi(D_{ii}^+)=\overline{(i,i+N)_r}=(i,i+N)_r$ and $\varphi(D_{ii}^-)=\overline{(i,i+N)_g}=(i,i+N)_g$. It is easy to prove $\varphi$ is an isomorphism and satisfy $\varphi\circ\tau_m=\tau_m\circ\varphi$.  By Theorem \ref{equi}, $\Delta(n,m)$ is isomorphic to the AR-quiver of $\C_{D_{n}}^{m}$, i.e. $\Delta(n,m)$ can be seen as a new combinatorial model of $\C_{D_{n}}^{m}$.
\end{proof}
 As a direct consequence of Proposition \ref{Pro1}, we have the following.
\begin{proposition}
There is a bijection between $m$-arcs $\overline{(i,j)}$ in $P$ and indecomposable objects in  $\mathrm{ind}\;\C_{D_{n}}^{m}$.
\end{proposition}
The bijection above induces a bijection between the subcategories of $\C_{D_{n}}^{m}$ and the sets of $m$-arcs in $P$ which are invariant
  under rotation of $180$ degrees. For a subcategory $\X$ of $\C_{D_{n}}^{m}$, we denote the corresponding set of $m$-diagonals in $P$ by $\XX$. For the paired $m$-arcs $\overline{(i,j)}$, we sometimes write $(i,j)$ without confusion in the rest of the paper.  For the $m$-arc $u=(i,j)$ in $P$, We denote the corresponding indecomposable object in $\C_{D_{n}}^{m}$ by $M_u$. We sometimes use an $m$-arc $(i,j)$ to represent an indecomposable object $M$ in $\C_{D_{n}}^{m}$ without confusion, denote by $M=(i,j)$.

\begin{example}
Suppose $m=2, n=4$. We draw the AR-quiver of $\C_{D_{4}}^{2}$ in a regular $14$-gon (see Fig. \ref{AR-quiver1} and compare with Fig. \ref{AR-quiver}).
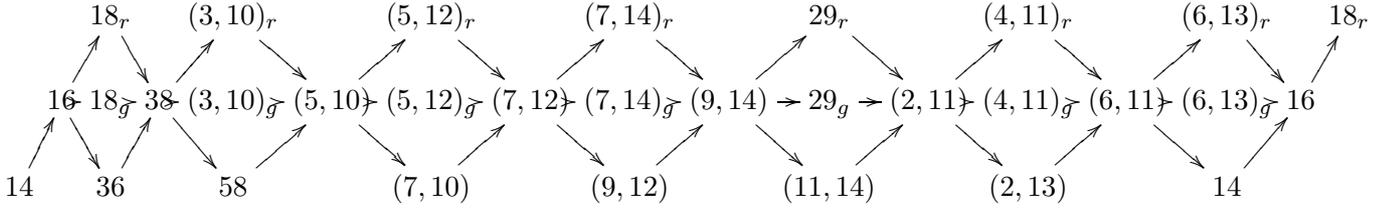
\begin{figure}
\centering
$$
{
\xymatrix@-7mm@C -0.2cm{
    &&18_{r}\ar[rdd] & & (3,10)_r  \ar[rdd] & & (5,12)_r \ar[rdd] && (7,14)_r \ar[rdd] && 29_r\ar[rdd] && (4,11)_r\ar[rdd] && (6,13)_r\ar[rdd] && 18_{r}\\
 \\
  &16 \ar[ruu]\ar[rdd]\ar[r]&18_{g}\ar[r]& 38 \ar[ruu]\ar[rdd]\ar[r]&(3,10)_g\ar[r]& (5,10)\ar[ruu]\ar[rdd]\ar[r]&(5,12)_g\ar[r]& (7,12)\ar[ruu]\ar[rdd]\ar[r]&(7,14)_g\ar[r]& (9,14)
  \ar[ruu]\ar[rdd]\ar[r]&29_g\ar[r]& (2,11)\ar[ruu]\ar[rdd]\ar[r]&(4,11)_g\ar[r]& (6,11) \ar[ruu]\ar[rdd]\ar[r]&(6,13)_g\ar[r]& 16\ar[ruu]\\
  \\
  14 \ar[ruu]&& 36 \ar[ruu] && 58 \ar[ruu]&& (7,10) \ar[ruu]&& (9,12) \ar[ruu]&& (11,14)\ar[ruu] && (2,13) \ar[ruu] && 14\ar[ruu]&&\\
\\
}
}
$$
\caption{Another description of the AR-quiver of  $\C_{D_{4}}^{2}$}
\label{AR-quiver1}
\end{figure}

\end{example}

\subsection{Torsion pairs in a triangulated category}
We recall some known facts about torsion pairs in a triangulated category based on \cite{BMRRT,IY,N}.

\begin{definition}\label{dftor}
Let $\X, \Y$ be subcategories of a triangulated category $\D$.
\begin{itemize}
  \item [(1)] The pair $(\X, \Y)$ is called a torsion pair if
  $$\Hom_{\D}(\X, \Y)=0\;\text{and}\;\D=\X\ast\Y. $$
  \item [(2)] The pair $(\X, \Y)$ is called a cotorsion pair if
  $$\Ext^{1}_{\D}(\X, \Y)=0\;\text{and}\;\D=\X\ast\Y[1]. $$
\end{itemize}
\end{definition}
\begin{remark}\label{remark3}
$(\X,\Y)$ is a torsion pair if and only if $(\X,\Y[-1])$ is a cotorsion pair.
\end{remark}

\begin{lemma}\cite[Proposition 2.3]{IY}\label{add}
Let $\X$ be a subcategory of $\C_{D_{n}}^{m}$. Then the following statements are equivalent.
\begin{itemize}
  \item [(1)] $\X$ is closed under extensions.
  \item [(2)] $(\X,\X^\perp)$ is a torsion pair.
  \item [(3)] $\X={^\perp(\X^\perp)}$.
\end{itemize}
\end{lemma}

\section{Ptolemy diagrams in  $P$ and torsion pairs in $\C_{D_{n}}^{m}$}
In this section, we define Ptolemy diagrams of type $D$ in $P$ and give a complete classification of torsion pairs in $\C_{D_{n}}^{m}$ when $m$ is odd. Firstly, we give the definition of crossing $m$-arcs in $P$. Note that the translation functor $\tau_m$ acts on diameter arcs is rotation by $m$ vertex in the counterclockwise direction and changing their colour if $m$ is odd. We suppose $m\in\mathbb{Z}_{\geq 1}$ is odd in the rest of the paper.

\subsection{Ptolemy diagrams in $P$}
\begin{definition}\label{c1}
\begin{enumerate}
\item[{(a)}] We say that two non-diameter $m$-arcs $(i,j)$ and
$(k,\ell)$ cross precisely if the elements $i,j,k,\ell$ are all
distinct and come in the order $i, k, j, \ell$
when moving around $P$ in one direction or the other
(i.e. counterclockwise or clockwise). In particular, the two
arcs in $\overline{(i,j)}$ do not cross.

Similarly, in the case $j=i+N$, the above condition defines
when a diameter arc $(i,i+N)_g$ (or $(i,i+N)_r$) crosses
the non-diameter $m$-arc $(k,\ell)$.
\item[{(b)}] We say that two non-diameter $m$-arcs $\overline{(i,j)}$ and
$\overline{(k,\ell)}$ cross if
there exist two arcs in these two pairs which cross in the sense of
part (a). (Note that then necessarily the other two rotated $m$-arcs also
cross.)

Similarly, the diameter $(i,i+N)_g$ (or $(i,i+N)_r$) crosses
the non-diameter arcs  $\overline{(k,\ell)}$  if it
crosses one of the arcs in $\overline{(k,\ell)}$.
(Note that it then necessarily crosses both arcs in $\overline{(k,\ell)}$.)
\item[{(c)}] Two different diameters $(i,i+N)_g$ and $(j,j+N)_r$ of different colour  cross, $j\not\in \{i,i+N\}$. But $(i,i+N)_g$ and $(i,i+N)_r$ do not cross. Moreover, any diameters of the same colour do not cross.
\end{enumerate}
\end{definition}
\begin{lemma}\cite[Proposition 2]{T}
Suppose $u, v$ are two $m$-arcs in $P$ which are invariant under rotation of $180$ degrees, $M_u, M_v$ are the corresponding indecomposable objects in $\C_{D_{n}}^{m}$. Then
\begin{itemize}
  \item [(1)] $u$ does not cross $v$ if and only if  $\Ext^{i}_{\C_{D_{n}}^{m}}(M_u,M_v)=0$ for all $1\leq i\leq m$.
  \item [(2)] $M_u[m]=M_{\tau_mu}$
\end{itemize}
\end{lemma}
Next we give a criterion to discriminate when an arc is an $m$-arc in $P$. Firstly, we need the following definition.
\begin{definition}\label{def:diatance}
Suppose $u=(i,j), v=(k,\ell)$ are two crossing $m$-arcs in $P$ with $i<k<j<\ell$ or $i<\ell<j<k$.
\begin{itemize}
  \item [(1)] Suppose $i<k<j<\ell$. Then we define the distance from $u$ to $v$ is $(k-i)$ mod $m$, denoted by $d(u,v)=(k-i)\;\mathrm{mod}\;m$ (expressed as a number between 1 and $m$). It is easy to see that $(k-i)$ mod $m$ equals $(\ell-j)$ mod $m$.
  \item [(2)] Suppose $i<\ell<j<k$. Then we define the distance from $u$ to $v$ is $(\ell-i)$ mod $m$ (which equals $(k-j)$ mod $m$), denoted by $d(u,v)=(\ell-i)\;\mathrm{mod}\;m$ (expressed as a number between 1 and $m$).
\end{itemize}

\end{definition}
Note that the definition of distance from $u$ to $v$ is compatible with \cite{T2} for type $A$. It is easy to see $d$ is not symmetrical even if it is called a ``distance". We have $d(u,v)+d(v,u)=m+1$, and $d(u,v)$ is equal to counting the number of elementary moves, clockwise and along the boundary, from $i$ to the closest endpoint of $v$ modulo $m$ (as a number in $[1,m]$). See Fig. \ref{0}.
\begin{figure}
\begin{center}
\begin{minipage}{.5\textwidth}
\centering
\begin{tikzpicture}[scale=2]
\begin{scope}
    \foreach \x in {-36,0,36,72,108,144,180,216,252,288,324}{
        \draw (\x:1 cm) -- (\x + 36: 1cm) -- cycle;
        \draw (72:1cm)--(252:1cm);
        \draw (0:1cm)--(180:1cm);
        \draw (0:1cm)--(180:1cm);
        \draw[dashed][->](0.5,1)arc(60:18:1.7);
        \node[right] at (0:1cm) {$k$};
         \node at (36:1cm) {};
         \node[above] at (72:1cm) {$i$};
         \node at (108:1cm) {$$};
         \node at (144:1cm) {$$};
         \node[left] at (180:1cm) {$\ell$};
         \node at (216:1cm) {$$};
         \node[below] at (252:1cm) {$j$};
         \node at (288:1cm) {$$};
         \node at (324:1cm) {$$};
}
\end{scope}
\end{tikzpicture}
\end{minipage}%
\begin{minipage}{.5\textwidth}
\centering
\begin{tikzpicture}[scale=2]
\begin{scope}
    \foreach \x in {-36,0,36,72,108,144,180,216,252,288,324}{
        \draw (\x:1 cm) -- (\x + 36: 1cm) -- cycle;
        \draw (72:1cm)--(252:1cm);
        \draw (0:1cm)--(180:1cm);
        \draw (0:1cm)--(180:1cm);
        \draw[dashed][->](1.1,-0.1)arc(-10:-90:1.3);
        \node[right] at (0:1cm) {$i$};
         \node at (36:1cm) {};
         \node[above] at (72:1cm) {$k$};
         \node at (108:1cm) {$$};
         \node at (144:1cm) {$$};
         \node[left] at (180:1cm) {$j$};
         \node at (216:1cm) {$$};
         \node[below] at (252:1cm) {$\ell$};
         \node at (288:1cm) {$$};
         \node at (324:1cm) {$$};
}
\end{scope}
\end{tikzpicture}
\end{minipage}
\end{center}
\caption{The distance from $u$ to $v$ is indicated by the dashed arrow, i.e, we have $d(u,v)=(k-i)$ mod m in the left figure , and $d(u,v)=(\ell-i)$ mod m in the right figure. }
\label{0}
\end{figure}
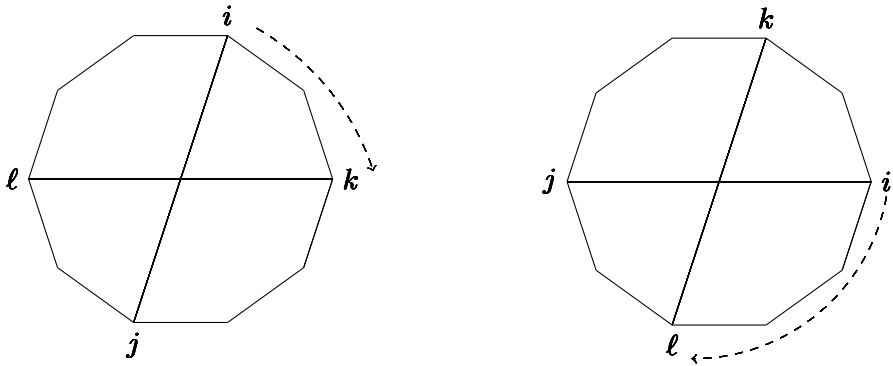
\begin{lemma}\label{B}
Suppose $u=(i,j), v=(k,\ell)$ are two crossing $m$-arcs in $P$. If $i<k<j<\ell$, then $(i,k)$ is an $m$-arc if and only if $d(u,v)=1$ and $k\neq i+1$. If $i<\ell <j<k$, then $(i,k)$ is an $m$-arc if and only if $d(v,u)=1$ and $k\neq i-1$.
\end{lemma}
\begin{proof}
This is a direct consequence of  Definition \ref{A} and the definition of the distance.
\end{proof}
The following Lemma in \cite{T2} for type $A$  also holds for type $D$. We establish the result without proof.
\begin{lemma}\cite[Proposition 8]{T2}\label{lemma0}
Suppose $u$ and $v$ are two $m$-arcs in $P$. Then $u$ and $v$ are crossing with $d(u,v)=j$ if and only if $\Ext^{j}_{\C_{D_{n}}^{m}}(M_u,M_v)\neq0$, for some $j\in\{1,\ldots,m\}$.
\end{lemma}

\begin{definition}\label{c2}
  Let $\X$ be a collection of $m$-arcs of $P$, which is invariant
  under rotation of $180$ degrees.  Then $\X$ is called a
  Ptolemy diagram of type $D$ if it satisfies the following
  conditions.  Let $\alpha = (i,j)$ and $\beta = (k,\ell)$
  be crossing $m$-arcs in $\X$ (in the sense of
  Definition \ref{c1}).
  \begin{enumerate}
  \item[{(Pt1)}] If $\alpha$ and $\beta$ are not diameters, then
  those of $(i,k)$, $(i,\ell)$, $(j,k)$, $(j,\ell)$ which are
    $m$-arcs of $P$ are also in $\X$.
    In particular, if two of the vertices $i,j,k,\ell$ are
    opposite vertices (i.e. one of $k$ and $\ell$ is equal to
    $i+N$ or $j+N$), then
    both the green and the red diameters connecting them are also
    in $\X$.
  \item[{(Pt2)}] If both $\alpha$ and $\beta$ are diameters
  (necessarily of different colours if $m$ is odd by Definition \ref{c1}\,(c)),
  then those of $(i,k)$, $(i,k+N)$, $(i+N,k)$, $(i+N,k+N)$ which are $m$-arcs of $P$ are also in $\X$.
\item[{(Pt3)}] If $\alpha$ is a diameter while $\beta$ is not a diameter,
  then those of $(i,k)$, $(i,\ell)$, $(j,k)$, $(j,\ell)$ which are
    $m$-arcs and do not cross the arc $(k+N,\ell +N)$ are also in
    $\X$. Additionally, the diameters $(k, k+N)$ and
    $(\ell, \ell +N)$ of the same colour as $\alpha$ are
    also in $\X$.
  \end{enumerate}

\end{definition}
These conditions are illustrated in Fig. \ref{10}.
\begin{figure}
  \centering
  \begin{enumerate}
  \item[{(Pt1)}]
  The first Ptolemy condition in type
    $D$:
$$
\begin{tikzpicture}[scale=2]
\begin{scope}
    \foreach \x in {-20,0,20,40,60,80,100,120,140,160,180,200,220,240,260,280,300,320}{
        \draw (\x:1 cm) -- (\x + 20: 1cm) -- cycle;
        \draw [thick](-20:1cm)--(80:1cm);
        \draw [thick](20:1cm)--(120:1cm);
        \draw [thick][dashed](-20:1cm)--(120:1cm);
        \draw [thick][dashed](20:1cm)--(80:1cm);
        \draw [thick][dashed](-20:1cm)--(20:1cm);
        \draw [thick][dashed](80:1cm)--(120:1cm);
        \draw [thick](160:1cm)--(260:1cm);
        \draw [thick](200:1cm)--(300:1cm);
        \draw [thick][dashed](160:1cm)--(300:1cm);
        \draw [thick][dashed](160:1cm)--(200:1cm);
        \draw [thick][dashed](260:1cm)--(300:1cm);
        \draw [thick][dashed](260:1cm)--(200:1cm);
        \node[right][below] at (-20:1cm) {$i$};
        \node[above] at (80:1cm) {$j$};
        \node[right] at (20:1cm) {$k$};
        \node[left][above] at (120:1cm) {$\ell$};
        \node[right][below] at (300:1cm) {$k+N$};
        \node[left] at (200:1cm) {$\ell+N$};
        \node[left] at (160:1cm) {$j+N$};
        \node[below] at (260:1cm) {$i+N$};
        }
\end{scope}
\end{tikzpicture}
\begin{tikzpicture}[scale=2][auto]
\begin{scope}
    \foreach \x in {-26,0,26,52,78,104,130,155,180,206,232,258,282,308}{
        \draw (\x:1 cm) -- (\x + 26: 1cm) -- cycle;
        \draw [thick](78:1cm)--(308:1cm);
        \draw [thick](26:1cm)--(258:1cm);
        \draw [thick](78:1cm)--(206:1cm);
        \draw [thick](258:1cm)--(130:1cm);
        \draw [thick][dashed](26:1cm)--(308:1cm);
        \draw [thick][dashed](206:1cm)--(130:1cm);
        \draw [thick][dashed](26:1cm)--(78:1cm);
        \draw [thick][dashed](258:1cm)--(308:1cm);
        \draw [thick][dashed](78:1cm)--(130:1cm);
        \draw [thick][dashed](206:1cm)--(258:1cm);
        \draw (78:1) edge[thick, dashed, decorate, decoration=snake, color=red!] (258:1);
\draw (78:1) edge[thick, dashed, color=green!] (258:1);
         \node[right] at (26:1cm) {$k$};
         \node[above] at (78:1cm) {$i$};
         \node[left] at (125:1cm) {$j+N$};
         \node[left] at (206:1cm) {$k+N$};
         \node[below] at (258:1cm) {$\ell=i+N$};
         \node[below] at (308:1cm) {$j$};

}
\end{scope}
\end{tikzpicture}
$$

  \item[{(Pt2)}]
The second Ptolemy condition in type
  $D$:
  $$
\begin{tikzpicture}[scale=2]
\begin{scope}
    \foreach \x in {-26,0,26,52,78,104,130,155,180,206,232,258,282,308}{
        \draw (\x:1 cm) -- (\x + 26: 1cm) -- cycle;
        \draw [thick][dashed](78:1cm)--(308:1cm);
        \draw [thick][dashed](258:1cm)--(130:1cm);
        \draw [thick][dashed](258:1cm)--(308:1cm);
        \draw [thick][dashed](78:1cm)--(130:1cm);
        \draw (78:1) edge[thick, decorate, decoration=snake, color=red!] (258:1);
\draw (308:1) edge[thick, color=green!] (130:1);
         \node[above] at (78:1cm) {$i$};
         \node[left] at (125:1cm) {$k+N$};
         \node[below] at (258:1cm) {$i+N$};
         \node[below] at (308:1cm) {$k$};
}
\end{scope}
\end{tikzpicture}
$$
\item[{(Pt3)}]
The third Ptolemy condition in type
  $D$:
\[
\begin{tikzpicture}[scale=2]
\begin{scope}
    \foreach \x in {-20,0,20,40,60,80,100,120,140,160,180,200,220,240,260,280,300,320}{
        \draw (\x:1 cm) -- (\x + 20: 1cm) -- cycle;
        \draw [thick](-20:1cm)--(80:1cm);
        \draw [thick](160:1cm)--(260:1cm);
        \draw [thick, color=green!](40:1cm)--(220:1cm);
        \draw [thick,dashed, color=green!](80:1cm)--(260:1cm);
        \draw [thick,dashed, color=green!](-20:1cm)--(160:1cm);
        \draw [thick][dashed](40:1cm)--(-20:1cm);
        \draw [thick][dashed](80:1cm)--(40:1cm);
        \draw [thick][dashed](160:1cm)--(220:1cm);
        \draw [thick][dashed](260:1cm)--(220:1cm);
        \node[right] at (-20:1cm) {$k$};
        \node[above] at (40:1cm) {$i$};;
        \node[left] at (220:1cm) {$i+N$};
        \node[above] at (80:1cm) {$\ell$};
        \node[left] at (160:1cm) {$k+N$};
        \node[below] at (260:1cm) {$\ell+N$};
        }
\end{scope}
\end{tikzpicture}
\]
\end{enumerate}
\caption{The Ptolemy conditions in type~$D$.}
\label{10}
\end{figure}

\begin{remark}
when $m=1$, Definition \ref{c2} coincides with Definition 2.1 in \cite{HJR3}. When $m>1$, We have the following results by Lemma \ref{B}:
\begin{itemize}
  \item [(1)] In Pt1, if none of the vertices $i,j,k,\ell$ are
    opposite vertices, then there are at most two of
  $(i,k)$, $(i,\ell)$, $(j,k)$, $(j,\ell)$  are
    $m$-arcs of $P$.
  \item [(2)] In Pt2, we have $\overline{(i,k)}=\overline{(i+N,k+N)}$, $\overline{(i,k+N)}=\overline{(i+N,k)}$. Moreover, if $(i,k)$ is an $m$-arc, then $(i,k+N)$ is not.
  \item [(3)] In Pt3, at most one of $(i,k),(i,\ell)$  is an $m$-arc.
\end{itemize}
\end{remark}

\begin{example}
Suppose $m=3, n=4$. Let $X$ be a Ptolemy diagram of type $D$ in $P$. (See Fig. \ref{20}).
\begin{itemize}
  \item [(1)] If $(1,5)$ and $(4,8)$ are two crossing non-diameter $3$-arcs in $X$, only $(1,8)$ is a $3$-arc in $X$, but $(1,4),(4,5)$ and $(5,8)$ are not $3$-arcs. Moreover, if $(1,8)$ and $(4,11)$ are two crossing non-diameter $3$-arcs in $X$ with $1$ and $11$ are opposite points, we have $(4,8)$ (non-diameter) and $(1,11)_{r,g}$ are (diameter) $3$-arcs in $X$, but $(1,4),(8,11)$ are not  $3$-arcs.
  \item [(2)] If $(1,11)_r$ and $(4,14)_g$ are two crossing diameter $3$-arcs in $X$, then$\overline{(1,14)}=\overline{(4,11)}$ is a $3$-arc in $X$, but $\overline{(1,4)}=\overline{(11,14)}$ is not a $3$-arc.
  \item [(3)] If $(1,5)$ and $(4,14)_r$ are two crossing $3$-arcs in $X$, then $(1,11)_r$ and $(5,15)_r$ are in $X$, but neither $(1,4)$ nor $(4,5)$ is a $3$-arc.
\end{itemize}

\begin{figure}
\begin{center}
\begin{tikzpicture}[scale=2]
\begin{scope}
    \foreach \x in {-18,0,18,36,54,72,90,108,126,144,162,180,198,216,234,252,270,288,306,324}{
        \draw (\x:1 cm) -- (\x + 18: 1cm) -- cycle;
        \draw [thick](-18:1cm)--(54:1cm);
        \draw [thick](0:1cm)--(288:1cm);
         \draw [thick](162:1cm)--(234:1cm);
        \draw [thick](180:1cm)--(108:1cm);
        \draw [thick][dashed](54:1cm)--(288:1cm);
        \draw [thick][dashed](234:1cm)--(108:1cm);
         \node[right] at (-18:1cm) {5};
         \node[right] at (0:1cm) {4};
         \node[right] at (18:1cm) {3};
         \node[right][above] at (36:1cm) {2};
         \node[above] at (54:1cm) {1};
         \node[above] at (72:1cm) {20};
         \node[above] at (90:1cm) {19};
         \node[left][above] at (108:1cm) {18};
         \node[left][above] at (126:1cm) {17};
         \node[left][above]  at (144:1cm) {16};
         \node[left] at (162:1cm) {15};
         \node[left] at (180:1cm) {14};
         \node[left] at (198:1cm) {13};
         \node[left] at (216:1cm) {12};
         \node[left] at (234:1cm) {11};
         \node[below] at (252:1cm) {10};
         \node[below] at (270:1cm) {9};
         \node[below] at (288:1cm) {8};
         \node[right][below] at (306:1cm) {7};
         \node[right] at (324:1cm) {6};

}
\end{scope}
\end{tikzpicture}
\begin{tikzpicture}[scale=2][auto]
\begin{scope}
    \foreach \x in {-18,0,18,36,54,72,90,108,126,144,162,180,198,216,234,252,270,288,306,324}{
        \draw (\x:1 cm) -- (\x + 18: 1cm) -- cycle;
        \draw [thick](234:1cm)--(0:1cm);
        \draw [thick](54:1cm)--(288:1cm);
         \draw [thick][dashed](0:1cm)--(288:1cm);
         \draw [thick](54:1cm)--(180:1cm);
        \draw [thick](234:1cm)--(108:1cm);
         \draw [thick][dashed](180:1cm)--(108:1cm);
        \draw (54:1) edge[thick, dashed, decorate, decoration=snake, color=red!] (234:1);
        \draw (54:1) edge[thick, dashed, color=green!] (234:1);
         \node[right] at (-18:1cm) {5};
         \node[right] at (0:1cm) {4};
         \node[right] at (18:1cm) {3};
         \node[right][above] at (36:1cm) {2};
         \node[above] at (54:1cm) {1};
         \node[above] at (72:1cm) {20};
         \node[above] at (90:1cm) {19};
         \node[left][above] at (108:1cm) {18};
         \node[left][above] at (126:1cm) {17};
         \node[left][above]  at (144:1cm) {16};
         \node[left] at (162:1cm) {15};
         \node[left] at (180:1cm) {14};
         \node[left] at (198:1cm) {13};
         \node[left] at (216:1cm) {12};
         \node[left] at (234:1cm) {11};
         \node[below] at (252:1cm) {10};
         \node[below] at (270:1cm) {9};
         \node[below] at (288:1cm) {8};
         \node[right][below] at (306:1cm) {7};
         \node[right] at (324:1cm) {6};

}
\end{scope}
\end{tikzpicture}
\begin{tikzpicture}[scale=2]
\begin{scope}
    \foreach \x in {-18,0,18,36,54,72,90,108,126,144,162,180,198,216,234,252,270,288,306,324}{
        \draw (\x:1 cm) -- (\x + 18: 1cm) -- cycle;
        \draw [thick][dashed](54:1cm)--(180:1cm);
        \draw [thick][dashed](234:1cm)--(0:1cm);
       \draw (234:1) edge[thick, decorate, decoration=snake, color=red!] (54:1);
       \draw (180:1) edge[thick, color=green!] (0:1);
         \node[right] at (-18:1cm) {5};
         \node[right] at (0:1cm) {4};
         \node[right] at (18:1cm) {3};
         \node[right][above] at (36:1cm) {2};
         \node[above] at (54:1cm) {1};
         \node[above] at (72:1cm) {20};
         \node[above] at (90:1cm) {19};
         \node[left][above] at (108:1cm) {18};
         \node[left][above] at (126:1cm) {17};
         \node[left][above]  at (144:1cm) {16};
         \node[left] at (162:1cm) {15};
         \node[left] at (180:1cm) {14};
         \node[left] at (198:1cm) {13};
         \node[left] at (216:1cm) {12};
         \node[left] at (234:1cm) {11};
         \node[below] at (252:1cm) {10};
         \node[below] at (270:1cm) {9};
         \node[below] at (288:1cm) {8};
         \node[right][below] at (306:1cm) {7};
         \node[right] at (324:1cm) {6};

}
\end{scope}
\end{tikzpicture}
\begin{tikzpicture}[scale=2][auto]
\begin{scope}
    \foreach \x in {-18,0,18,36,54,72,90,108,126,144,162,180,198,216,234,252,270,288,306,324}{
        \draw (\x:1 cm) -- (\x + 18: 1cm) -- cycle;
        \draw [thick](-18:1cm)--(54:1cm);
        \draw [thick](162:1cm)--(234:1cm);
       \draw (0:1) edge[thick, decorate, decoration=snake, color=red!] (180:1);
       \draw (-18:1) edge[thick, dashed, decorate, decoration=snake, color=red!] (162:1);
       \draw (54:1) edge[thick, dashed, decorate, decoration=snake, color=red!] (234:1);
         \node[right] at (-18:1cm) {5};
         \node[right] at (0:1cm) {4};
         \node[right] at (18:1cm) {3};
         \node[right][above] at (36:1cm) {2};
         \node[above] at (54:1cm) {1};
         \node[above] at (72:1cm) {20};
         \node[above] at (90:1cm) {19};
         \node[left][above] at (108:1cm) {18};
         \node[left][above] at (126:1cm) {17};
         \node[left][above]  at (144:1cm) {16};
         \node[left] at (162:1cm) {15};
         \node[left] at (180:1cm) {14};
         \node[left] at (198:1cm) {13};
         \node[left] at (216:1cm) {12};
         \node[left] at (234:1cm) {11};
         \node[below] at (252:1cm) {10};
         \node[below] at (270:1cm) {9};
         \node[below] at (288:1cm) {8};
         \node[right][below] at (306:1cm) {7};
         \node[right] at (324:1cm) {6};

}
\end{scope}
\end{tikzpicture}
\caption{Ptolemy diagrams of type $D$ in a regular $20$-gon}
\label{20}
\end{center}
\end{figure}
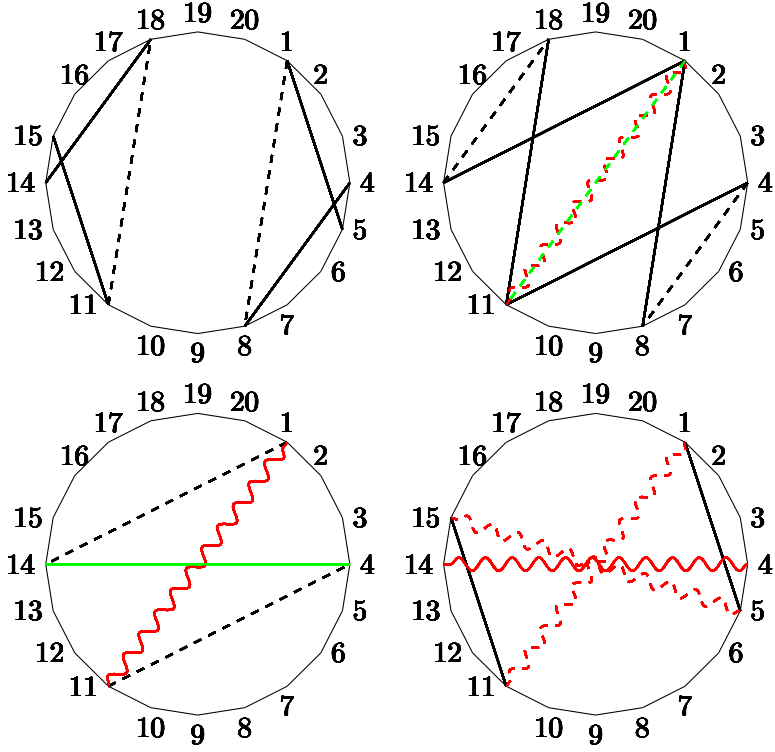

\end{example}

Let $\mathfrak{U}$ be a set of $m$-arcs in $P$. We define $\mathfrak{U}^\bot$ to be the set consisting of two kinds of $m$-arcs:
\begin{itemize}
  \item [(1)] $u\in P$ such that $u$ does not cross any $m$-arcs in $\mathfrak{U}$, or
  \item [(2)] $u\in P$ such that any $m$-arc $v\in\mathfrak{U}$ crossing $u$ satisfies  $d(v,u)>1$.
\end{itemize}
Similarly, we define $^\bot\mathfrak{U}$ to be the set consisting of two kinds of $m$-arcs:
\begin{itemize}
  \item [(1)] $u\in P$ such that $u$ does not cross any $m$-arcs in $\mathfrak{U}$, or
  \item [(2)] $u\in P$ such that any $m$-arc $v\in\mathfrak{U}$ crossing $u$ satisfies  $d(u,v)>1$.
\end{itemize}

Note that for an arbitrary $m$-arc $(a,b)$ in $P$, since $\overline{(a,b)}=\overline{(a+N,b+N)}$, we can always choose a representative $(a,b)\in\overline{(a,b)}$ such that $a<b$ and $b-a\leq N$.
\begin{lemma}\label{lemma3}
Let $\mathfrak{U}$ be a set of $m$-arcs in $\Pi$. Then $\mathfrak{U}^\bot$ and $^\bot\mathfrak{U}$ are Ptolemy diagrams of type $D$.
\end{lemma}
\begin{proof}
We only need to show $\mathfrak{U}^\bot$ is a Ptolemy diagram of type $D$, since $^\bot\mathfrak{U}$ is a Ptolemy diagram of type $D$ can be proved similarly. Suppose $u=(i,j)$ and $v=(k,\ell)$ are two crossing $m$-arcs in $\mathfrak{U}^\bot$, and suppose $(i,k)$ is an $m$-arc, we need to show $(i,k)\in\mathfrak{U}^\bot$.

Suppose $u=(i,j)$ and $v=(k,\ell)$ are non-diameter $m$-arcs. If $(i,k)\not\in\mathfrak{U}^\bot$, then there exists an $m$-arc $w\in\mathfrak{U}$ such that $w$ crosses $(i,k)$ and $d(w,(i,k))=1$. Suppose $w=(a,b)$. Then $d(w,(i,k))=1$ implies one can pivot the node $a$ until it meets $k$ or $i$ in a number $x=1$(mod $m$) of steps. Note that every $m$-arc $w\in\mathfrak{U}$  crossing $(i,k)$ crosses either $u$ or $v$, then $d(w,(i,k))=1$ implies that $w$ crosses $u$ and $d(w,u)=1$, or $w$ crosses $v$ and $d(w,v)=1$, a contradiction with $u,v\in\mathfrak{U}^\bot$. The case when $u=(i,j)$ or $v=(k,\ell)$ is a diameter, or both are diameters can be proved similarly.
\end{proof}

\subsection{Classification of cotorsion pairs in $\C_{D_{n}}^{m}$ }

We give a geometric realization of torsion pairs in $\C_{D_{n}}^{m}$ when $m$ is odd in this subsection.

For any vertex $a$ in $P$, set
$$(a,-)=\{(a,b)|(a,b)\text{\;is\;an\;$m$-arc\;in\;}P\}.$$
Specifically, $(a,-)$ is the set of $m$-arc in $P$ with $a$ as a node. In order to get the main result of this paper, we need the following lemmas and proposition.

\begin{lemma}\label{lemma1}
Suppose $\mathfrak{U}$ is a set of $m$-arcs in $P$ and $(a,c)$ is an $m$-arc in $^\bot(\mathfrak{U}^\bot)$. Then any $m$-arc $v$  crossing $(a,c)$ and satisfying $d((a,c),v)=1$ is not contained in $\mathfrak{U}^\bot$. Especially, any $m$-arc $v$ in $(a+1,-), (c+1,-)$ that crosses $(a,c)$  is not contained in $\mathfrak{U}^\bot$.
\end{lemma}
\begin{proof}
If there exists an $m$-arc $v\in\mathfrak{U}^\bot$ crossing $(a,c)$ and $d((a,c),v)=1$, then $(a,c)$ is not contained in $^\bot(\mathfrak{U}^\bot)$ by Lemma \ref{lemma0}, a contradiction.

Suppose $(a+1,b)$ is an $m$-arc in $(a+1,-)$ and $(a+1,b)$ crosses $(a,c)$. Since $d((a,c), (a+1,b))=1$, $(a+1,b)$ is not contained in $\mathfrak{U}^\bot$.
\end{proof}
Since $(a,-)$ is a finite set, there exists a maximal integer $b$ and a minimal integer $c$ such that $(a,b)\in(a,-)$ and $(a,c)\in(a,-)$. The idea of ``minimal" and ``maximal" will appear in the following proofs.
\begin{lemma}\label{lemma2}
Suppose $\mathfrak{U}$ is a Ptolemy diagram of type $D$ and $a$ is a vertex in $P$ satisfying $(a,-)\cap\mathfrak{U}\neq\emptyset$. Suppose $(a,a+N)_{r,g}\notin\mathfrak{U}$ and $b(\neq a-m-1)$ is maximal such that $(a,b)\in\mathfrak{U}$. Then $(a-m,b)$ is an $m$-arc in $\mathfrak{U}^\bot$. Especially, if $(a-m,b)$ is a diameter, then $(a-m,b)_{r,g}\in\mathfrak{U}^\bot$.
\end{lemma}
\begin{proof}
The fact $b\neq a-m-1$  implies $(a-m,b)$ is an $m$-arc. If $(a-m,b)\not\in\mathfrak{U}^\bot$, then there exists an $m$-arc (maybe a diameter) $u\in\mathfrak{U}$ crossing $(a-m,b)$ and $d(u,(a-m,b))=1$. Suppose $(a-m,b)$ is non-diameter. We have the following two cases:
\begin{itemize}
  \item [(1)] The $m$-arc $u=(s,t)$ is non-diameter. The fact $u=(s,t)$  crosses $(a-m,b)$ and $d(u,(a-m,b))=1$ implies $a-m<a<s<b<t$. Then we have $(a,b)$ crosses $u=(s,t)$ in $\mathfrak{U}$. Since  $\mathfrak{U}$ is a Ptolemy diagram of type $D$ and $d((s,t),(a-m,b))=1$ , $(a,t)\in\mathfrak{U}$ is an $m$-arc by Pt1. The fact $(a,t)\in\mathfrak{U}$ with $b<t$ is in contradiction with the maximality of $b$.
  \item [(2)] The $m$-arc $u=(s,s+N)$ is a diameter. Then $u=(s,s+N)$ crosses $(a-m,b)$ and $d((s,s+N),(a-m,b))=1$ implies $a-m<a<s<b$. Then we have $(a,b)$ crosses the diameter $u=(s,s+N)$ in $\mathfrak{U}$. Since  $\mathfrak{U}$ is a Ptolemy diagram of type $D$ and $d((s,s+N),(a-m,b))=1$ , $(a,a+N)$ with the same color as $u=(s,s+N)$ is in $\mathfrak{U}$ by Pt3. This is a contradiction with the fact $(a,a+N)_{r,g}\notin\mathfrak{U}$.
\end{itemize}
Then we prove that $(a-m,b)$ is an $m$-arc in $\mathfrak{U}^\bot$. If $(a-m,b)$ is a diameter, then $(a-m,b)_{r,g}\in\mathfrak{U}^\bot$ can be proved similarly.
\end{proof}
We give a "minimal" version of Lemma \ref{lemma2} without proof.
\begin{lemma}\label{lemma4}
Suppose $\mathfrak{U}$ is a Ptolemy diagram of type $D$ and $a$ is a vertex in $P$ satisfying $(a,-)\cap\mathfrak{U}\neq\emptyset$. Suppose  $(a,a+N)_{r,g}\notin\mathfrak{U}$ and $b(\neq m+a+1)$ is minimal such that $(a,b)\in\mathfrak{U}$. Then $(a+1,b-m+1)$ is an $m$-arc in $\mathfrak{U}^\bot$. Especially, if $(a+1,b-m+1)$ is a diameter, then $(a+1,b-m+1)_{r,g}\in\mathfrak{U}^\bot$.
\end{lemma}

\begin{lemma}\label{lemma5}
Suppose $\mathfrak{U}$ is a Ptolemy diagram of type $D$ and $a$ is a vertex in $P$ satisfying $(a,-)\cap\mathfrak{U}\neq\emptyset$. Suppose $a+N$ is  minimal such that $(a,a+N)_r\in\mathfrak{U}$ (resp. $(a,a+N)_g\in\mathfrak{U}$), but $(a,a+N)_g\notin\mathfrak{U}$ (resp. $(a,a+N)_r\notin\mathfrak{U}$). Then $(a+1,a+N+1)_{r}\in\mathfrak{U}^\bot$ (resp. $(a+1,a+N+1)_{g}\in\mathfrak{U}^\bot$).
\end{lemma}
\begin{proof}
If $(a+1,a+N+1)_{r}\notin\mathfrak{U}^\bot$, then there exists an $m$-arc $u\in\mathfrak{U}$ crossing $(a+1,a+N+1)_r$ and $d(u,(a+1,a+N+1)_r)=1$. There are the following two cases.
\begin{itemize}
  \item [(1)] If $u$ is a diameter, then $u=(s,s+N)_g$ with $a+1<s<a+N$ since $(a,a+N)_g\notin\mathfrak{U}$. This implies $(s,s+N)_g$ crosses $(a,a+N)_r$, and $d(u,(a+1,a+N+1)_r)=1$ implies $(a,s)$ is an $m$-arc. Since $\mathfrak{U}$ is a Ptolemy diagram of type $D$, we have $(a,s)\in\mathfrak{U}$ by Pt2. This is a contradiction with the minimality of $a+N$.
  \item [(2)] If $u=(s,t)$ is a non-diameter $m$-arc, then $a+1<t<a+N$. This implies $u=(s,t)$ crosses $(a,a+N)_r$, and $d(u,(a+1,a+N+1)_r)=1$ implies $(a,t)$ is an $m$-arc. Since $\mathfrak{U}$ is a Ptolemy diagram of type $D$, we have $(a,t)\in\mathfrak{U}$ by Pt3. This is a contradiction with the minimality of $a+N$. 
\end{itemize}
So we get the result $(a+1,a+N+1)_{r}\in\mathfrak{U}^\bot$.
\end{proof}
We give a "maximal" version of Lemma \ref{lemma5} without proof.
\begin{lemma}\label{lemma6}
Suppose $\mathfrak{U}$ is a Ptolemy diagram of type $D$ and $a$ is a vertex in $P$ satisfying $(a,-)\cap\mathfrak{U}\neq\emptyset$. Suppose $a+N$ is  maximal such that $(a,a+N)_r\in\mathfrak{U}$ (resp. $(a,a+N)_g\in\mathfrak{U}$), but $(a,a+N)_g\notin\mathfrak{U}$ (resp. $(a,a+N)_r\notin\mathfrak{U}$). Then $(a-m,a+N-m)_{r}\in\mathfrak{U}^\bot$ (resp. $(a-m,a+N-m)_{g}\in\mathfrak{U}^\bot$).
\end{lemma}
\begin{proposition}\label{prop1}
$\mathfrak{U}$ is a Ptolemy diagram of type $D$ if and only if $\mathfrak{U}={^\bot(\mathfrak{U}^\bot)}$.
\end{proposition}
\begin{proof}
Suppose $\mathfrak{U}={^\bot(\mathfrak{U}^\bot)}$. Then $\mathfrak{U}$ is a Ptolemy diagram of type $D$ by Lemma \ref{lemma3}.

Now suppose $\mathfrak{U}$ is a Ptolemy diagram of type $D$. The inclusion $\mathfrak{U}\subseteq {^\bot(\mathfrak{U}^\bot)}$ is clear. Suppose $u=(a,b)$ is an $m$-arc in $^\bot(\mathfrak{U}^\bot)$, we need to show $u\in\mathfrak{U}$.

Firstly, we show $(a,-)\cap\mathfrak{U}\neq\emptyset$ and $(b,-)\cap\mathfrak{U}\neq\emptyset$.

Since $(a-m,a+1)$ is an $m$-arc crossing $u=(a,b)$ and $d((a,b),(a-m,a+1))=1$, we have $(a-m,a+1)\not\in\mathfrak{U}^\bot$ by Lemma \ref{lemma1}. Then there exists an $m$-arc $w\in\mathfrak{U}$ such that $w$ crosses $(a-m,a+1)$ and $d(w,(a-m,a+1))=1$, which implies $(a,-)\cap\mathfrak{U}\neq\emptyset$. The fact $(b,-)\cap\mathfrak{U}\neq\emptyset$ can be proved similarly.

We have the following two cases:
\begin{itemize}
  \item [(1)] Suppose $u=(a,a+N)$ is a diameter. Without loss of generality, let $u=(a,a+N)_r$.

We have got $(a,-)\cap\mathfrak{U}\neq\emptyset$, we show $(a,-)\cap\mathfrak{U}$ contains diameter arcs. Suppose, on the contrary, $(a,-)\cap\mathfrak{U}$ does not contain diameter arcs, i.e. $(a,a+N)_r\notin\mathfrak{U}$ and $(a,a+N)_g\notin\mathfrak{U}$, we need to show there is a contradiction.

Suppose $c$ is maximal such that $(a,c)\in\mathfrak{U}$. If $c<a+N$, then $(a-m,c)$ is an $m$-arc in $\mathfrak{U}^\bot$ by Lemma \ref{lemma2}. Since $(a-m,c)$ crosses $(a,a+N)_r$ and $d((a,a+N)_r,(a-m,c))=1$, this is a contradiction with the facts $(a-m,c)\in\mathfrak{U}^\bot$ and $(a,a+N)_r\in {^\bot(\mathfrak{U}^\bot)}$. So we have that there exist $m$-arcs $(a,c)$  in $\mathfrak{U}$ with $c>a+N$. Similarly, there exist $m$-arcs $(a,d)$  in $\mathfrak{U}$ with $d<a+N$. Suppose $d$ is maximal such that $(a,d)\in\mathfrak{U}$ with $d<a+N$, and suppose $c$ is minimal such that $(a,c)\in\mathfrak{U}$ with $c>a+N$. If $a+N<c<d+N$, then $(a+N,d+N)$ crosses $(a,c)$ in $\mathfrak{U}$, which implies $(a,a+N)_r,g\in\mathfrak{U}$ by Pt1, a contradiction with assumption. So let $c>d+N$. Now consider the $m$-arc $(d,c-m+1)$ (If $(d,c-m+1)$ is a diameter, we consider $(d,c-m+1)_g$). Since $(d,c-m+1)$ is an $m$-arc crossing $(a,a+N)_r$, and satisfying $d((a,a+N)_r,(d,c-m+1))=1$. Then $(d,c-m+1)\not\in\mathfrak{U}^\bot$ by Lemma \ref{lemma1}, so there exists $w\in\mathfrak{U}$ such that $w$ crosses $(d,c-m+1)$ and $d(w,(d,c-m+1))=1$. There are the following two cases.
\begin{itemize}
  \item [(i)] If $w=(s,t)$ is a non-diameter $m$-arc, then $d<s<c-m+1$. Since $d(w,(d,c-m+1))=1$, we have $(a,s)$ is an $m$-arc. Because $\mathfrak{U}$ is a Ptolemy diagram of type $D$, we have $(a,s)\in\mathfrak{U}$ by Pt1, this is a contradiction with the maximality of $d$ or the minimality of $c$.
  \item [(ii)] If  $w$ is a diameter, then $w=(s,s+N)$ for an integer $s$. Without loss of generality, suppose $w=(s,s+N)_r$. Then $(s,s+N)_r$ crosses $(a,c)$ or $(a,d)$. Since $\mathfrak{U}$ is a Ptolemy diagram of type $D$, we have $(a,a+N)_r\in\mathfrak{U}$ by Pt3, a contradiction with assumption.
\end{itemize}
So we have $(a,-)\cap\mathfrak{U}$ contains diameter arcs. Suppose $(a,a+N)_g\in\mathfrak{U}$, but $(a,a+N)_r\notin\mathfrak{U}$, we show there is a contradiction. Then we complete the proof of $(a,a+N)_r\in\mathfrak{U}$.

If $a+N$ is maximal (resp. minimal) such that $(a,a+N)_g\in\mathfrak{U}$, then $(a-m,a+N-m)_{g}\in\mathfrak{U}^\bot$ by Lemma \ref{lemma6} (resp. $(a+1,a+N+1)_{g}\in\mathfrak{U}^\bot$ by Lemma \ref{lemma5}). Since $(a-m,a+N-m)_{g}$ (resp. $(a+1,a+N+1)_{g}$) crosses $(a,a+N)_r\in{^\bot(\mathfrak{U}^\bot)}$ and $d((a,a+N)_r,(a-m,a+N-m)_{g})=1$ (resp. $d((a,a+N)_r,(a+1,a+N+1)_{g})=1$),  this is a contradiction by Lemma \ref{lemma1}. Then there exist $m$-arcs $(a,c)$  in $\mathfrak{U}$ with $c>a+N$ and $m$-arcs $(a,d)$  in $\mathfrak{U}$ with $d<a+N$. Suppose $d$ is maximal such that $(a,d)\in\mathfrak{U}$ with $d<a+N$, and suppose $c$ is minimal such that $(a,c)\in\mathfrak{U}$ with $c>a+N$. If $a+N<c<d+N$, then $(a+N,d+N)$ crosses $(a,c)$ in $\mathfrak{U}$, which implies $(a,a+N)_{r,g}\in\mathfrak{U}$ by Pt1, a contradiction with assumption. So let $c>d+N$. Now consider the diameter $(c-m+1,c+N-m+1)_g$. On one hand, since the diameter $(c-m+1,c+N-m+1)_g$ crosses $(a,a+N)_r$ and satisfies $d((a,a+N)_r,(c-m+1,c+N-m+1)_g)=1$, $(c-m+1,c+N-m+1)_g\not\in\mathfrak{U}^\bot$ by Lemma \ref{lemma1}. On the other hand, if $(c-m+1,c+N-m+1)_g\not\in\mathfrak{U}^\bot$, then there exists an $m$-arc $w\in\mathfrak{U}$ such that $w$ crosses $(c-m+1,c+N-m+1)_g$ and $d(w,(c-m+1,c+N-m+1)_g)=1$. There are the following two cases.
\begin{itemize}
  \item [(i)] If $w=(s,t)$ is a non-diameter $m$-arc, then $c+N-m+1<s<c-m+1+2N$ (mod $2N$). In this case, $w$ crosses $(a,c)$ or $(a,d)$ or $(a,a+N)_g$. Since $d(w,(d,c-m+1))=1$, we have $(a,s)$ is an $m$-arc. Because $\mathfrak{U}$ is a Ptolemy of type $D$, we have $(a,s)\in\mathfrak{U}$ by Pt1 or Pt3, this is a contradiction with the maximality of $d$ or the minimality of $c$.
  \item [(ii)] If $w$ is a diameter, then $w=(s,s+N)_r$ with $c+N-m+1<t<c-m+1+2N$ (mod $2N$). Since $(s,s+N)_r$ crosses $(a,a+N)_g\in\mathfrak{U}$ and $d((s,s+N)_r,(c-m+1,c+N-m+1)_g)=1$, we have $(a,s)$ is an $m$-arc. Since $\mathfrak{U}$ is a Ptolemy of type $D$, we have $(a,s)\in\mathfrak{U}$ by Pt2, a contradiction with the maximality of $d$ or the minimality of $c$.
\end{itemize}
Thus, we prove there is a contradiction.

\item [(2)] Suppose $u=(a,b)$, with $a<b<a+N$, is a non-diameter $m$-arc.

We have got the fact that $(a,-)\cap\mathfrak{U}\neq\emptyset$ and $(b,-)\cap\mathfrak{U}\neq\emptyset$. We divide into the following three cases based on whether $(a,a+N)\in\mathfrak{U}$ or not.
\begin{itemize}
  \item [(i)] Suppose $(a,a+N)_r\notin\mathfrak{U}$ and $(a,a+N)_g\notin\mathfrak{U}$. Then we can prove $(a,b)\in\mathfrak{U}$ as (1).
  \item [(ii)] Suppose both $(a,a+N)_r\in\mathfrak{U}$ and $(a,a+N)_g\in\mathfrak{U}$. If the diameter $(b,b+N)\in\mathfrak{U}$ (red or green), then $(a,a+N)$ and $(b,b+N)$ of different colors  cross in $\mathfrak{U}$. Since $\mathfrak{U}$ is a Ptolemy of type $D$, we have $(a,b)\in\mathfrak{U}$ by Pt2. If $(b,b+N)_r\notin\mathfrak{U}$ and $(b,b+N)_g\notin\mathfrak{U}$, then we can prove $(a,b)\in\mathfrak{U}$ as (1).
    \item [(iii)] Suppose $(a,a+N)_r\in\mathfrak{U}$, but $(a,a+N)_g\notin\mathfrak{U}$, the case $(a,a+N)_g\in\mathfrak{U}$, but $(a,a+N)_r\notin\mathfrak{U}$, can be proved similarly. If $(b,-)\cap\mathfrak{U}$ does not contain diameter arcs, then we can prove $(a,b)\in\mathfrak{U}$ as (1). If $(b,b+N)_g\in\mathfrak{U}$, then $(a,a+N)_r$ crosse $(b,b+N)_g$ in $\mathfrak{U}$, and we have $(a,b)\in\mathfrak{U}$ by Pt2. Now suppose $(b,b+N)_r\in\mathfrak{U}$ but $(b,b+N)_g\notin\mathfrak{U}$. If there exists $c$ such that $(a,c)\in\mathfrak{U}$ with $b\leq c<a+N$, then $(a,c)$ (except $c=b$) crosses $(b,b+N)_r$ in $\mathfrak{U}$. Since $\mathfrak{U}$ is a Ptolemy of type $D$, we have $(a,b)\in\mathfrak{U}$ by Pt3. Suppose there are no $m$-arcs $(a,c)\in\mathfrak{U}$ with $b\leq c<a+N$. We will show there is a contradiction, and then we complete our proof.

        Suppose $d$ is maximal such that $(a,d)\in\mathfrak{U}$ with $d<b$, if no such $d$ exists, let $d=a+1$. Now we consider the diameter $(d,d+N)_r$. Since $(d,d+N)_r$ crosses $(a,b)$ and $d((a,b),(d,d+N)_r)=1$, we have $(d,d+N)_r\notin\mathfrak{U}^\bot$ by Lemma \ref{lemma1}. Then there exists an $m$-arc $w$ in $\mathfrak{U}$ such that $w$ crosses $(d,d+N)_r$ and $d(w,(d,d+N)_r)=1$. Firstly, suppose $w=(s,s+N)_g$ is a diameter with $d<s<d+N$, then $(s,s+N)_g$ crosses $(a,a+N)_r$, and $(s,s+N)_g$ crosses $(a,d)$ if $a+N<s<d+N$. Since $\mathfrak{U}$ is a Ptolemy of type $D$ and $d(w,(d,d+N)_r)=1$, we have $(a,s)\in\mathfrak{U}$ with $d<s<a+N$ by Pt2, or $(a,a+N)_g\in\mathfrak{U}$ when $a+N<s<d+N$ by Pt2, a contradiction with the assumption that there are no $m$-arcs $(a,c)\in\mathfrak{U}$ with $b\leq c<a+N$ and $(a,a+N)_g\notin\mathfrak{U}$, or the maximality of $d$. Now suppose $w=(s,t)$, a non-diameter arc with $d<t<d+N$. Then $(s,t)$ crosses $(a,d)$ when $a<s<d$, or crosses $(a,a+N)_r$ when $s<a$. Since $\mathfrak{U}$ is a Ptolemy of type $D$ and $d(w,(d,d+N)_r)=1$, we have $(a,t)$ is an $m$-arc and $(a,t)\in\mathfrak{U}$, a contradiction with the assumption that there are no $m$-arcs $(a,c)\in\mathfrak{U}$ with $b\leq c<a+N$, or the maximality of $d$.        
\end{itemize}

\end{itemize}

\end{proof}
Now we get the  main result of this paper.
\begin{theorem}\label{mainresult}
Let $\X$ be a subcategory of $\C_{D_{n}}^{m}$, and $\XX$ be the corresponding set of $m$-arcs in $P$. Then the following statements are equivalent.
\begin{itemize}
  \item [(1)] $\X$ is closed under extensions.
  \item [(2)] $(\X,\X^\perp)$ is a torsion pair.
  \item [(3)] $\X={^\perp(\X^\perp)}$.
  \item [(4)] $\XX$ is a Ptolemy diagram of type $D$.
\end{itemize}
\end{theorem}
\begin{proof}
The equivalence of (1)-(3) is obvious by Lemma \ref{add}. By Lemma \ref{lemma0}, the subcategory $\X^\perp[-1]$ (resp. $^\perp\X[1]$) corresponds to $\XX^\perp$ (resp. $^\perp\XX$). The equivalence of (3) and (4) is derived from Lemma \ref{lemma3} and Proposition \ref{prop1}.
\end{proof}
\section{Applications}

\subsection{Classification of $m$-rigid objects and $m$-cluster tilting objects in $\C_{D_{n}}^{m}$ }
Firstly, we recall the definition of $m$-rigid subcategories and $m$-cluster tilting subcategories in $\C_{D_{n}}^{m}$.

\begin{definition}[\cite{KR,T,Z,IY,J}]
Let $\X$ be a subcategory of $\C_{D_{n}}^{m}$.
\begin{itemize}
  \item [(1)] The subcategory $\X$ is called $m$-rigid if $\Ext_{\C_{D_{n}}^{m}}^{i}(\X,\X)=0$ for all $1\leq i\leq m$. The subcategory $\X$ is called maximal $m$-rigid if $\X$ is maximal with respect to this property, i.e., if $\Ext_{\C_{D_{n}}^{m}}^{i}(\X\oplus\add M,\X\oplus\add M)=0$ for all $1\leq i\leq m$, then $M\in\X$. An object $X$ is called an $m$-rigid object if $\add X$ is $m$-rigid, and $X$ is called a maximal $m$-rigid object if $\add X$ is $m$-maximal rigid.
  \item [(2)] $\X$ is called $m$-cluster tilting if $X\in\X$ if and only if
      $\Ext_{\C_{D_{n}}^{m}}^{i}(X,\X)=0$ for all $1\leq i\leq m$ (or equivalently, if and only if $\Ext_{\C_{D_{n}}^{m}}^{i}(\X,X)=0$ for all $1\leq i\leq m$).
      An object $X$ is called an $m$-cluster tilting object if $\add\;X$ is $m$-cluster tilting.
\end{itemize}
We point out the fact that for an object $X$ of $\C_{D_{n}}^{m}$, $X$ is $m$-cluster tilting if and only if $X$ is  maximal $m$-rigid \cite[Theorem 3.3]{ZZ}.
\end{definition}
\begin{definition}[\cite{J}]
Let $\mathfrak{U}$ be a set of $m$-arcs in $P$.
\begin{itemize}
  \item [(1)] $\mathfrak{U}$ is called a partial $(m+2)$-angulation of $P$ if $\mathfrak{U}$ is a set of non-crossing $m$-arcs in $P$.
  \item [(2)] $\mathfrak{U}$ is called an $(m+2)$-angulation of $P$ if $\mathfrak{U}$ is a maximal set of non-crossing $m$-arcs in $P$.
\end{itemize}
\end{definition}
\begin{proposition}\label{prop2}
Let $\X$ be a subcategory of $\C_{D_{n}}^{m}$, and $\XX$ be the corresponding set of $m$-arcs in $P$. Then the following statements are equivalent.
\begin{itemize}
  \item [(1)] $\X$ is $m$-rigid.
  \item [(2)] $\XX$ is a partial $(m+2)$-angulation .
  \item [(3)] $(\X,\X^\perp)$ is a torsion pair such that $\X[1]\ast\ldots\ast \X[m]\subseteq\X^\perp$.
\end{itemize}
\end{proposition}
\begin{proof}
The equivalence of $(1)$ and $(2)$ is clear by Lemma \ref{lemma0}. Now we prove the equivalence of $(1)$ and $(3)$.

Suppose $\X$ is $m$-rigid. Then $\XX$ is a Ptolemy diagram of type $D$, so $(\X,\X^\perp)$ is a torsion pair by Theorem \ref{mainresult}. Since $\Ext^i_{\C_{D_{n}}^{m}}(\X,\X)=0=\Hom_{\C_{D_{n}}^{m}}(\X,\X[i])$ for all $1\leq i\leq m$, $\X[i]\subseteq\X^\perp$ for all $i$, thus $\X[1]\ast\ldots\ast \X[m]\subseteq\X^\perp$.

Conversely, if $(\X,\X^\perp)$ is a torsion pair such that $\X[1]\ast\ldots\ast \X[m]\subseteq\X^\perp$, then $\Ext_{\C_{D_{n}}^{m}}^i(\X,\X)=\Hom_{\C_{D_{n}}^{m}}(\X,\X[i])=0$ for all $1\leq i\leq m$, so $\X$ is $m$-rigid.
\end{proof}
By the proposition above, we classify $m$-cluster tilting objects of $\C_{D_{n}}^{m}$. Note that classification of $m$-cluster tilting objects and $m$-cluster tilting subcategories are equivalent for $\C_{D_{n}}^{m}$.
\begin{corollary}
	Let $\X$ be a subcategory of $\C_{D_{n}}^{m}$, and $\XX$ be the corresponding set of $m$-arcs in $P$. Then the following statements are equivalent.
\begin{itemize}
  \item [(1)] $\X$ is $m$-cluster tilting.
  \item [(2)] $\XX$ is an $(m+2)$-angulation of $P$.
  \item [(3)] $(\X,\X[1]\ast\ldots\ast \X[m])$ is a torsion pair.
\end{itemize}
\end{corollary}
\begin{proof}
The equivalence of $(1)$ and $(2)$ is clear, which was obtained in \cite[Theorem 4.5]{J}. Now we prove the equivalence of $(1)$ and $(3)$.

Suppose $\X$ is an $m$-cluster tilting subcategory. Then $(\X,\X[1]\ast\ldots\ast \X[m])$ is a torsion pair by Theorem 3.1 in \cite{IY}, so $\X[1]\ast\ldots\ast \X[m]=\X^\perp$.

Conversely, suppose $(\X,\X\ast\ldots\ast \X[m-1])$ is a torsion pair. Then we have $\C_{D_{n}}^{m}=\X\ast\ldots\ast \X[m]$ and $\X$ is $m$-rigid by Proposition \ref{prop2}, $\X$ is an $m$-cluster tilting subcategory is a direct consequence of Lemma 3.9 in \cite{ZZ2}.
\end{proof}

\subsection{Relationship with the cluster category of type $D_{n}$}
When $m=1$, $\C_{D_{n}}^{1}$ (denote by $\C_{D_{n}}$) is the classical cluster category of type $D_{n}$. In this subsection, we use our classification of torsion pairs in
$\C_{D_{n}}^{m}$ to recover one main result in \cite{HJR3}, which gives a correspondence between torsion pairs and Ptolemy diagrams of type $D$.

\begin{corollary}\cite[Theorem 1.1]{HJR3}
Let $\X$ be a subcategory of $\C_{D_{n}}$, and $\XX$ be the corresponding set of arcs in the $2n$-gon. Then the following statements are equivalent:
\begin{itemize}
  \item [(1)] $(\X,\X^\perp)$ is a torsion pair in $\C_{D_{n}}$.
  \item [(2)] $\XX$ is a Ptolemy diagram of type $D$.
\end{itemize}
\end{corollary}
\begin{proof}
This is a direct consequence of Theorem \ref{mainresult}.
\end{proof}

\bigskip


\end{document}